\RequirePackage{fix-cm}
\documentclass[smallextended]{svjour3}       
\smartqed  
\usepackage{graphicx}

\usepackage{algorithm}
\usepackage{algpseudocode}
\newtheorem{assumption}{Assumption}
\usepackage{amsmath}
\usepackage{amssymb}
\usepackage{mathrsfs}
\usepackage{mathtools}
\usepackage{graphicx}
\usepackage{hyperref}

\begin{document}

\title{CPQR-based randomized algorithms for generalized CUR decompositions
}


\author{Guihua Zhang         \and Hanyu Li \and Yimin Wei
}


\institute{Guihua Zhang and Hanyu Li \at
              College of Mathematics and Statistics, Chongqing University, Chongqing, 401331, 
              People’s Republic of China.
              \email{lihy.hy@gmail.com; hyli@cqu.edu.cn}\\
Yimin Wei  \at
              School of Mathematical Sciences and Shanghai Key Laboratory of Contemporary Applied Mathematics, Fudan University, Shanghai 200433, People’s Republic of China.
              \email{ymwei@fudan.edu.cn; yimin.wei@gmail.com} 
}

\date{Received: date / Accepted: date}

\maketitle

\begin{abstract}
Based on the column pivoted QR decomposition, we propose some randomized algorithms including pass-efficient ones for the generalized CUR decompositions of matrix pair and matrix triplet. 
Detailed error analyses of these algorithms are provided.  Numerical experiments are given to test the proposed randomized algorithms.
\end{abstract}

\section{Introduction}
As we know, the singular value decomposition (SVD) is the most popular tool for 
reducing dimension of a data matrix. However, 
the final features 
with the dimension reduction by SVD are usually difficult to explain and analyse \cite{mahoney2009cur,gidisu2022generalized}. To this end, some scholars proposed the CUR decomposition, 
which is composed of  subsets of the original columns and rows of the given matrix. Specifically, 
for a matrix $A$, the CUR decomposition can be expressed as
\begin{align*}
	A \approx C_AM_AR_A,
\end{align*}
where $C_A$ and $ R_A$ are the subsets of columns and rows of $A$ respectively, and $M_A$ can be constructed by different ways; see Remark \ref{remark1} below. Recently, Gidisu and Hochstenbach \cite{gidisu2022generalized,gidisu2022rsvd} proposed the generalized CUR decompositions for matrix pair and matrix triplet; see Definitions \ref{definition} and \ref{definition1} below. 

At present, there are many works on applications (i.e., \cite{gidisu2022generalized,gidisu2022rsvd,cai2021robust,lei2022exemplar}), algorithms (i.e., \cite{gidisu2022generalized,gidisu2022rsvd,wang2013improving,sorensen2016deim,voronin2017efficient,chen2020efficient,dong2021simpler,boutsidis2014optimal,drineas2008relative}), and perturbation analyses (i.e., \cite{hamm2021perturbations,chen2022tensor,che2022perturbations}) for (generalized) CUR decompositions. Here we mainly focus on their algorithms. For the case of CUR decomposition, Drineas et al. \cite{drineas2008relative} presented the relative error guaranteed algorithm from the point of view of leverage-score sampling.  
Subsequently, Wang and Zhang \cite{wang2013improving} proposed the adaptive sampling algorithm, 
and Boutsidis and Woodruff \cite{boutsidis2014optimal} combined the previous two sampling methods with BSS sampling \cite{boutsidis2014near} 
and obtained the optimal CUR decomposition.
Later, Sorensen and Embree \cite{sorensen2016deim} introduced a novel algorithm by means of the discrete empirical interpolation method (DEIM), and Voronin and Martinsson \cite{voronin2017efficient} presented efficient algorithms based on the column pivoted QR decomposition (CPQR). Recently, Chen et al. \cite{chen2020efficient} proposed an algorithm by using the truncated LU factorization, and Dong and Martinsson \cite{dong2021simpler} compared the randomized algorithms based on CPQR \cite{voronin2017efficient} and LU decomposition with partial pivoting \cite{dong2021simpler}.
For the case of generalized CUR decompositions of matrix pair and matrix triplet, by virtue of the DEIM, Gidisu and Hochstenbach \cite{gidisu2022generalized,gidisu2022rsvd} proposed the corresponding deterministic algorithms. 
Very recently, Cao et al. \cite{cao2023randomized} presented some randomized algorithms based on L-DEIM \cite{Gidisu2022hybrid}. 

In this paper, we investigate the CPQR-based randomized algorithms for the generalized CUR decompositions of matrix pair and matrix triplet, and 
present the corresponding expectation error bounds. 

The rest of this paper is organized as follows. In Section \ref{sec-pre}, we introduce some preliminaries including the definitions of the generalized CUR decompositions for matrix pair and matrix triplet and the interpolative decomposition. The CPQR-based randomized algorithms and their error bounds are presented in Sections \ref{sec-pairs} and \ref{sec-triplets}. Finally, we provide some numerical results to test the proposed algorithms. 

\section{Preliminaries} \label{sec-pre}
We first introduce some notations used in this paper. Given a matrix $A\in \mathbb{R}^{m \times n}$, $A^T$, $A^{\dagger}$, $\Vert A\Vert$ , $\Vert A\Vert_{F}$ and $\sigma_j(A)$ denote its transpose, Moore-Penrose inverse, spectral norm, Frobenius norm and $j$th largest singular value, respectively. 
And, we use $A(I,J)$ to denote the part of $A$ consisting of rows from the index set $I$ and columns from the index set $J$. 
In addition, let $[n]=\{1,\cdots,n\}$, $I_n$ be the $n\times n$ identity matrix, $\mathbb{E}$ be the expectation notation, $\mathbb{E}\left[ \cdot \vert \cdot \right]$ be the conditional expectation, 
$P_A$ be the orthogonal projector on ${\rm range}(A)$, and $P_{A,B}$ be the oblique projector on ${\rm range}(A)$ along ${\rm range}(B)$. 

\subsection{CUR decomposition for matrix pair}

\begin{definition}[\cite{gidisu2022generalized}]\label{definition}
	Let $(A,B)$ be a matrix pair, where $A\in \mathbb{R}^{m \times n}$ and $B\in \mathbb{R}^{d \times n}$. The CUR decomposition of $(A,B)$ of rank $k$ is
	\begin{align*}
		A_k=C_AM_AR_A,
		B_k=C_BM_BR_B,
	\end{align*}
where $C_A=A(:,J)\in \mathbb{R}^{m \times k}, C_B=B(:,J)\in \mathbb{R}^{d \times k}, R_A=A(I_A,:)\in \mathbb{R}^{k \times n}$ and $R_B=B(I_B,:)\in \mathbb{R}^{k \times n}$.
\end{definition}

\begin{remark}\label{remark1}
    The middle matrices $M_A$ and $M_B$ have some different ways to compute. For example, 
    we can construct $M_A=A(I,J)^\dagger$, or $M_A=C_A^\dagger AR_A^\dagger$ \cite{hamm2021perturbations}. In this paper, we use $M_A=C_A^\dagger AR_A^\dagger$ to analyse the related algorithms. 
\end{remark}

To obtain the error analyses of our randomized algorithms of generalized CUR decompositions, 
we  need 
the generalized singular value decomposition (GSVD). 
For the above matrix pair $(A,B)$, 
the GSVD can be expressed as
\begin{align}
	A=U\Sigma_{1}Y^T, B=\tilde{V}\Sigma_{2}Y^T, \label{gsvd}
\end{align}
where $U\in \mathbb{R}^{m \times m}$ and $ \tilde{V}\in \mathbb{R}^{d \times d}$ are column orthogonal, $Y\in \mathbb{R}^{n \times n}$ is nonsingular, and $\Sigma_1 \in \mathbb{R}^{m \times n}$ and $ \Sigma_2 \in \mathbb{R}^{d \times n}$ are diagonal with diagonal entries being in $[0,1]$. Here, we let the diagonal entries of $\Sigma_1$ be in nondecreasing order, while those of $\Sigma_2$ be in nonincreasing order. For the details of GSVD, please refer to \cite{van1976generalizing,van1985computing}.

\subsection{CUR decomposition for matrix triplet}
\begin{definition}[\cite{gidisu2022rsvd}] \label{definition1}
	Let $(A,B,G)$ be a matrix triplet, where $A\in \mathbb{R}^{m \times n}$, $B\in \mathbb{R}^{m \times t}$ and $G\in \mathbb{R}^{d \times n}$. The CUR decomposition of $(A,B,G)$ of rank $k$ is
	\begin{align*}
		A_k=C_AM_AR_A,
		B_k=C_BM_BR_B,
		G_k=C_GM_GR_G,
	\end{align*}
	where $C_A=A(:,J)\in \mathbb{R}^{m \times k}, C_B=B(:,J_B)\in \mathbb{R}^{m \times k}, 
	C_G=G(:,J)\in \mathbb{R}^{d \times k}, R_A=A(I,:)\in \mathbb{R}^{k \times n}, R_B=B(I,:)\in \mathbb{R}^{k \times t}$ and 
	$R_G=G(I_G,:)\in \mathbb{R}^{k \times n}$.
\end{definition}
\begin{remark}
	The decomposition shows that the selected column indices of $A$ and $G$ are the same, and the selected row indices of $A$ and $B$ are the same. For the middle matrices $M_A, M_B$ and $M_G$, they can be constructed as discussed in Remark \ref{remark1}.
\end{remark}

\subsection{Interpolative decomposition (ID)}
The ID 
is very close to the CUR decomposition. Specifically, given $A\in \mathbb{R}^{m \times n}$, it has the following forms
\begin{align*}
	A\approx \hat{C}\hat{V}^T, A\approx \hat{W}\hat{R},
\end{align*}
where $\hat{C}\in \mathbb{R}^{m \times k}$ and $\hat{R}\in \mathbb{R}^{k \times n}$ consist of $k$ columns 
and rows of $A$, respectively, and 
$\hat{V}\in \mathbb{R}^{n \times k}$ and $\hat{W}\in \mathbb{R}^{k \times n}$ are such that ${\rm max}_{i,j}\lvert \hat{V}(i,j)\rvert \leq 1$ and ${\rm max}_{i,j}\lvert \hat{W}(i,j)\rvert \leq 1$, respectively. The above two IDs are one-sided IDs, i.e., column ID and row ID. Accordingly, there exists a two-sided ID: 
\begin{align*}
	A\approx \hat{W}A_s\hat{V}^T,
\end{align*}
where $\hat{W}\in \mathbb{R}^{m \times k}$, $\hat{V}\in \mathbb{R}^{n \times k}$, and $A_s\in \mathbb{R}^{k \times k}$ 
is the submatrix of $A$.

\section{Randomized algorithms for CUR decomposition of matrix pair} \label{sec-pairs}
We first propose a randomized algorithm based on CPQR. Then,  
a pass-efficient version of the 
algorithm is derived. Their  expectation error bounds are also presented correspondingly.
Moreover, we also provide the alternative error analyses of algorithms on the basis of GSVD.

\subsection{CPQR-based randomized algorithm}
For $A\in \mathbb{R}^{m \times n}$ and $B\in \mathbb{R}^{d \times n}$, assume ${\rm rank}(A) > k$ and $ {\rm rank}(B) > k$. Using CPQR, 
we have
\begin{align*}
	\begin{bmatrix}
		A \\ B
	\end{bmatrix}\hat{\Pi}=\hat{Q}\hat{T}=\hat{Q}\begin{bmatrix}
	\hat{T}_1, \hat{T}_2
\end{bmatrix},
\end{align*} 
where $\hat{Q}\in \mathbb{R}^{(m+d) \times (m+d)}$ is orthogonal, $\hat{T}\in \mathbb{R}^{(m+d) \times n}$ is upper triangular with  $\hat{T}_1\in \mathbb{R}^{(m+d) \times k}$ and $\hat{T}_2\in \mathbb{R}^{(m+d) \times (n-k)}$, and $\hat{\Pi} \in \mathbb{R}^{n \times n}$ is a permutation matrix. 
Similar to \cite{voronin2017efficient}, $C_A$ and $C_B$ in the CUR decomposition of the matrix pair $(A,B)$ can be obtained from 
$\hat{Q}\hat{T}_1$ directly. 
However, it 
is not very efficient on the computation cost. To tackle this problem, we consider the randomized CPQR 
\cite{dong2021simpler,duersch2017randomized,duersch2020randomized,martinsson2017householder}. 

Let $\Omega \in \mathbb{R}^{l \times (m+d)}$ be a Gaussian matrix, where $l=k+p$ and $p$ is the oversampling parameter. Then, using CPQR, we have
\begin{align*}
	\Omega \begin{bmatrix}
		A \\ B
	\end{bmatrix}\Pi=\tilde{Q}T=\tilde{Q}\begin{bmatrix}
	T_1, T_2
\end{bmatrix},
\end{align*}
where $\tilde{Q}\in \mathbb{R}^{l \times l}$ is orthogonal, $T_1\in \mathbb{R}^{l \times l}$ is invertible and upper triangular, and $\Pi=I_n(:,J_n)$ with $J_n$ being a permuted index vector. Based on the above decomposition, we can get the column index vector $J$ appearing in the CUR decomposition of the matrix pair $(A,B)$, 
i.e., $J=J_n(1:l)$. 
The specific algorithm is listed 
in Algorithm \ref{alg1}, from which, we can obtain
\begin{align*}
	XI_n(:,J)=X\Pi_C=\tilde{Q}T_1.
\end{align*}

\begin{algorithm}[ht] 
	\caption{Column selection for the CUR decomposition of $(A,B)$} 
	\label{alg1}
	\hspace*{0.02in} {\bf Input:} 
	 $A\in \mathbb{R}^{m \times n}$, $B\in \mathbb{R}^{d \times n}$, the target rank $k <n$, and the sample size $l=k+p$.\\
	\hspace*{0.02in} {\bf Output:} 
	A column index vector $J$ with $\lvert J\rvert =l$, $C_A$ and $C_B$.
	\begin{algorithmic}[1]
		\State Draw a Gaussian matrix $\Omega \in \mathbb{R}^{l \times (m+d)}$.
				
		\State Compute $X=\Omega \begin{bmatrix}
			A \\ B
		\end{bmatrix}\in \mathbb{R}^{l \times n}$.
		\State Compute the CPQR of $X$, that is
		\begin{align*}
			XI_n(:,J_n)=X\Pi = X\begin{bmatrix}
				\Pi_C, \Pi_C^c
			\end{bmatrix}=\tilde{Q}\begin{bmatrix}
				T_1, T_2
			\end{bmatrix},
		\end{align*}
where $\Pi_C \in \mathbb{R}^{n \times l}$ and $\Pi_C^c \in \mathbb{R}^{n \times (n-l)}$.
		\State $J=J_n(1:l)$, $C_A=A(:,J)$ and $C_B=B(:,J)$.	
	\end{algorithmic}
\end{algorithm}

To investigate the theoretical analysis of Algorithm \ref{alg1}, we further consider the following column ID of the matrix pair $(A,B)$: 
\begin{align*}
	\begin{bmatrix}
		A\\ B
	\end{bmatrix} \approx \begin{bmatrix}
	C_A \\ C_B
\end{bmatrix} V^T =\begin{bmatrix}
C_AV^T \\ C_BV^T
\end{bmatrix},
\end{align*}
which 
can be written as $A \approx C_AV^T, B \approx C_BV^T$.
Solving one of them by the least squares method implies
\begin{align}
	V^T=C_A^{\dagger}A=C_B^{\dagger}B.
\end{align}
Besides, noting that $X\Pi_C$ is invertible due to both $\tilde{Q}$ and $T_1$ being invertible, we define a rank-$l$ oblique projector onto the row space of $X$ as follows:
\begin{align*}
 P_{\Pi_C,N_1}=\Pi_C(X\Pi_C)^{-1}X=\Pi_CT_1^{-1}\tilde{Q}^TX.
\end{align*}
Since we mainly focus on the range where the oblique projector projects on,  hereafter we just use $N_x$ to denote the along range ${\rm range}(N_x)$ for oblique projectors. 
It is clear that $XP_{\Pi_C,N_1}=X$.

Now, we present the
theoretical analysis of Algorithm \ref{alg1}. 
\begin{theorem} \label{theorem1}
	Let $C_A\in \mathbb{R}^{m \times l}$ and $C_B\in \mathbb{R}^{d \times l}$ be obtained by Algorithm \ref{alg1}, and $V^T=C_A^{\dagger}A=C_B^{\dagger}B$. Then 
	\begin{align}
		\Vert A-C_AV^T \Vert \leq \Vert I_n-P_{\Pi_C,N_1} \Vert \Vert A(I_n-X^{\dagger}X) \Vert,  \label{column A} \\ 
		\Vert B-C_BV^T \Vert \leq \Vert I_n-P_{\Pi_C,N_1} \Vert \Vert B(I_n-X^{\dagger}X) \Vert. \label{column B}
	\end{align}
\end{theorem}
\begin{proof}
	We first consider the case for the matrix $A$. Since $C_A=A(:,J)=A\Pi_C$ is full column rank, we find that
	\begin{align*}
		C_AC_A^{\dagger}A=A\Pi_C(C_A^TC_A)^{-1}C_A^TA=AP_{\Pi_C,N_2},
	\end{align*}
where $P_{\Pi_C,N_2}=\Pi_C(C_A^TC_A)^{-1}C_A^TA$. Considering the relationship between $P_{\Pi_C,N_2}$ and $P_{\Pi_C,N_1}$, 
\begin{align*}
	P_{\Pi_C,N_2}P_{\Pi_C,N_1}&=\Pi_C(C_A^TC_A)^{-1}C_A^TA\Pi_C(X\Pi_C)^{-1}X \\ &=\Pi_C(X\Pi_C)^{-1}X=P_{\Pi_C,N_1}, 
\end{align*}
we get
\begin{align*}
 A(I_n-P_{\Pi_C,N_2})&=A(I_n-P_{\Pi_C,N_2})(I_n-P_{\Pi_C,N_1}) \\ &=(I_m-C_AC_A^{\dagger})A(I_n-P_{\Pi_C,N_1}).
\end{align*}
Hence, 
\begin{align*}
 \Vert A-C_AV^T \Vert&=\Vert (I_m-C_AC_A^{\dagger})A \Vert = \Vert A(I_n-P_{\Pi_C,N_2}) \Vert \\ &\leq \Vert I_m-C_AC_A^{\dagger} \Vert \Vert A(I_n-P_{\Pi_C,N_1}) \Vert  \\
	&=\Vert A(I_n-P_{\Pi_C,N_1}) \Vert 
	=\Vert A(I_n-X^{\dagger}X)(I_n-P_{\Pi_C,N_1}) \Vert \\
	& \leq 	\Vert A(I_n-X^{\dagger}X) \Vert \Vert I_n-P_{\Pi_C,N_1} \Vert,
\end{align*}
where the third equality holds because $I_m-C_AC_A^{\dagger}$ is an orthogonal projector and the last equality relies on 
$XP_{\Pi_C,N_1}=X$.

For the case of the matrix $B$, we can similarly define 
$P_{\Pi_C,N_3}=\Pi_C(C_B^TC_B)^{-1}C_B^TB$
and get 
$P_{\Pi_C,N_3}P_{\Pi_C,N_1}=P_{\Pi_C,N_1}$.
Thus, along the similar line, (\ref{column B}) can be obtained.
\end{proof}

Next, we 
discuss how to select the rows of $A$ and $B$, respectively. The idea is to find 
the row index vectors $I_A\subset [m]$ and $I_B\subset [d]$ 
from $C_A$ and $C_B$. 
To this end, we consider the exact CPQRs of the transposes of $C_A$ and $C_B$:
\begin{align*}
	\underset{l\times m}{C_A^T}\begin{bmatrix}
		\underset{m\times l}{\Pi_{Ar}}, \underset{m\times (m-l)}{\Pi_{Ar}^c}
	\end{bmatrix}&=\underset{l\times l}{Q_A} \begin{bmatrix}
	\underset{l\times l}{T_{A1}}, \underset{l\times (m-l)}{T_{A2}}
\end{bmatrix}=\underset{l\times l}{A_s^T}\begin{bmatrix}
\underset{l\times l}{I_l}, \underset{l\times (m-l)}{T_{A1}^{-1}T_{A2}}
\end{bmatrix}, \\
\underset{l\times d}{C_B^T}\begin{bmatrix}
	\underset{d\times l}{\Pi_{Br}}, \underset{d\times (d-l)}{\Pi_{Br}^c}
\end{bmatrix}&=\underset{l\times l}{Q_B} \begin{bmatrix}
	\underset{l\times l}{T_{B1}}, \underset{l\times (d-l)}{T_{B2}}
\end{bmatrix}=\underset{l\times l}{B_s^T}\begin{bmatrix}
	\underset{l\times l}{I_l}, \underset{l\times (d-l)}{T_{B1}^{-1}T_{B2}}
\end{bmatrix},
\end{align*}
where $\Pi_{Ar}=I_m(:,I_A)$, $\Pi_{Br}=I_d(:,I_B)$, 
$Q_A$ and $ Q_B$ are orthogonal, 
and $T_{A1}$ and $ T_{B1}$ are invertible and upper triangular due to $C_A$ and $C_B$ being full column rank. 

Further, the exact row IDs of $C_A$ and $C_B$ are given by
\begin{align*}
	C_A=\Pi_{Ar}(Q_AT_{A1})^T:=W_AA_s,\  C_B=\Pi_{Br}(Q_BT_{B1})^T:=W_BB_s.
\end{align*}
Here, $A_s=(Q_AT_{A1})^T$ and $B_s=(Q_BT_{B1})^T$ consist of the rows of $C_A$ and $C_B$, respectively.
By denoting $R_A=A_sV^T$ and $R_B=B_sV^T$, we have
\begin{align*}
	W_AA_sV^T=C_AV^T=C_AC_A^{\dagger}A=W_AR_A=(	W_AA_s)A_s^{\dagger}(A_sV^T)=C_AA_s^{\dagger}R_A, \\ 
	W_BB_sV^T=C_BV^T=C_BC_B^{\dagger}B=W_BR_B=(	W_BB_s)B_s^{\dagger}(B_sV^T)=C_BB_s^{\dagger}R_B.
\end{align*}
Thus, combining with the results that $C_A$ and $ C_B$ are full column rank and $R_A$ and $ R_B$ are full row rank implies 
\begin{align*}
	A_s^{\dagger}=C_A^{\dagger}AR_A^{\dagger}:=M_A,\  B_s^{\dagger}=C_B^{\dagger}BR_B^{\dagger}:=M_B.
\end{align*}

Therefore, we can have 
the randomized algorithm for the CUR decomposition of matrix pair, i.e., 
Algorithm \ref{alg2}. 

\begin{algorithm}[ht] 
	\caption{Randomized algorithm for the CUR decomposition of $(A,B)$} 
	\label{alg2}
	\hspace*{0.02in} {\bf Input:} 
	$A\in \mathbb{R}^{m \times n}$, $B\in \mathbb{R}^{d \times n}$, the target rank $k <n$, and the sample size $l=k+p$.\\
	\hspace*{0.02in} {\bf Output:} 
	$C_A$, $R_A$, $M_A$, 
	$C_B$, $R_B$, and $M_B$.
	\begin{algorithmic}[1]
		\State Perform Algorithm \ref{alg1} on the matrix pair $(A,B)$ to obtain $C_A=A(:,J)$ and $C_B=B(:,J)$ with $\lvert J\rvert =l$.
		
		\State Compute the exact CPQRs of $C_A^T$ and $C_B^T$, respectively, to obtain $R_A=A(I_A,:)$ and $R_B=B(I_B,:)$ with $\lvert I_A\rvert=\lvert I_B\rvert =l$.
		\State 	Compute $M_A=C_A^{\dagger}AR_A^{\dagger}$ and  $M_B=C_B^{\dagger}BR_B^{\dagger}$.
	\end{algorithmic}
\end{algorithm}

To present the expectation error bounds of Algorithm \ref{alg2}, the following lemma is necessary. 

\begin{lemma} \cite[Theorem 10.5 and 10.6]{halko2011finding} \label{lem-expectation}
	Given a matrix $A\in \mathbb{R}^{m \times n}$, the target rank $k$ and the oversampling parameter $p$ satisfying $k+p \leq {\rm min}\{m,n\}$, and a standard Gaussian matrix $\Omega \in \mathbb{R}^{(k+p) \times m}$, the column sketch $X=\Omega A$ satisfies
	\begin{align*}
		&\mathbb{E}\Vert  A(I_n-X^\dagger X)  \Vert \leq 
		\left( 1+\sqrt{\frac{k}{p-1}} \right) \sigma_{k+1}(A) +\frac{{\rm e}\sqrt{k+p}}{p}\left( \sum_{j>k}^{}\sigma_j^2(A) \right)^{1/2}, \\
		&\mathbb{E}\Vert  A(I_n-X^\dagger X)  \Vert_F \leq \left( 1+\sqrt{\frac{k}{p-1}} \right)^{1/2}\left( \sum_{j>k}^{}\sigma_j^2(A) \right)^{1/2}.
	\end{align*}
\end{lemma}  



\begin{theorem} \label{theorem-err1}
	Let $C_A\in \mathbb{R}^{m \times l}, C_B\in \mathbb{R}^{d \times l}, R_A\in \mathbb{R}^{l \times n}, R_B\in \mathbb{R}^{l \times n}, M_A\in \mathbb{R}^{l \times l}$, and $M_B\in \mathbb{R}^{l \times l}$ be obtained by Algorithm \ref{alg2}. Let $k$ be the target rank and $p$ be the oversampling parameter. 
	Then 
\begin{align*}
	{\rm max}\{ \mathbb{E}\Vert A-&C_AM_AR_A \Vert, \mathbb{E}\Vert B-C_BM_BR_B \Vert \} \leq \sqrt{1+(n-k-p)4^{k+p-1}}    \\
& \times\left[ \left( 1+\sqrt{\frac{k}{p-1}} \right) \sigma_{k+1}\begin{bmatrix}
	A \\ B
\end{bmatrix} +\frac{{\rm e}\sqrt{k+p}}{p}\left( \sum_{j>k}^{}\sigma_j^2\begin{bmatrix}
A \\ B
\end{bmatrix} \right)^{1/2} \right], \\
 	{\rm max}\{ \mathbb{E}\Vert A-&C_AM_AR_A \Vert_F, \mathbb{E}\Vert B-C_BM_BR_B \Vert_F \} \leq
 	\sqrt{1+(n-k-p)4^{k+p-1}}    \\
 	&\times\left( 1+\sqrt{\frac{k}{p-1}} \right)^{1/2}\left( \sum_{j>k}^{}\sigma_j^2\begin{bmatrix}
 		A \\ B
 	\end{bmatrix} \right)^{1/2}.
\end{align*}
\end{theorem}
\begin{proof}
	By virtue of  Lemma 3.2 in \cite{dong2021simpler}, we have 
	\begin{align*}
  \Vert I_n-P_{\Pi_C,N_1} \Vert \leq \sqrt{1+(n-k-p)4^{k+p-1}}.
	\end{align*}
	Thus, combining the fact 
	\begin{align*}
		&{\rm max}\{\Vert A(I_n-X^{\dagger}X) \Vert, \Vert B(I_n-X^{\dagger}X) \Vert\} \leq  \left \Vert \begin{bmatrix}
			A \\B
		\end{bmatrix}(I_n-X^{\dagger}X) \right \Vert
	\end{align*}
	with Theorem \ref{theorem1} and Lemma \ref{lem-expectation} 
	implies the desired results. 
\end{proof}

\begin{remark}
	The probability error bounds for Algorithm \ref{alg2} can be similarly obtained by using \cite[Theorem 10.7 and 10.8]{halko2011finding}. For brevity, we omit them.
\end{remark}

\subsection{Pass-efficient randomized algorithm}
Considering that 
the matrix pair $(A,B)$ may be very expensive to visit 
in practise, 
we now investigate the pass-efficient randomized algorithm for CUR decomposition. The idea is to introduce an extra random matrix to find the row index sets $I_A$ and $I_B$. Specifically, let $\Omega_1\in \mathbb{R}^{l \times n}$ be a Gaussian matrix and set $Y_1=A\Omega_1^T$ and $ Y_2=B\Omega_1^T$. Then, by the CPQRs of $Y_1^T$ and $Y_2^T$, we have 
\begin{align*}
	Y_1^T\Pi_1=Y_1^T \begin{bmatrix}
		\Pi_{1r}, \Pi_{1r}^c
	\end{bmatrix}=
	Q_1 \begin{bmatrix}
		T_{11}, T_{12} 
	\end{bmatrix}, \\
Y_2^T\Pi_2=Y_2^T \begin{bmatrix}
	\Pi_{2r},
 \Pi_{2r}^c
\end{bmatrix}=
Q_2 \begin{bmatrix}
	T_{21}, T_{22} 
\end{bmatrix},
\end{align*}
where $\Pi_{1r}=I_m(:,I_A)$ and  $\Pi_{2r}=I_d(:,I_B)$. The specific  algorithm is summarized in Algorithm \ref{alg3}.   

\begin{algorithm}[ht] 
	\caption{Pass-efficient algorithm for the CUR decomposition of $(A,B)$} 
	\label{alg3}
	\hspace*{0.02in} {\bf Input:} 
	$A\in \mathbb{R}^{m \times n}$, $B\in \mathbb{R}^{d \times n}$, the target rank $k <n$, and the sample size $l=k+p$.\\
	\hspace*{0.02in} {\bf Output:} 
	$C_A$, $R_A$, $M_A$, 
	$C_B$, $R_B$, and $M_B$. 
	\begin{algorithmic}[1]
		\State Draw two Gaussian matrices $\Omega\in \mathbb{R}^{l \times (m+d)}$ and $\Omega_1\in \mathbb{R}^{l \times n}$.
		\State In a single pass, compute $X=\Omega \begin{bmatrix}
				A \\B
			\end{bmatrix}, Y_1=A\Omega_1^T$, and $ Y_2=B\Omega_1^T$.
	    \State Compute the exact CPQR of $X$ to obtain $C_A=A(:,J)$ and $C_B=B(:,J)$ with $\lvert J\rvert =l$.	
		\State Compute the exact CPQRs of $Y_1^T$ and $Y_2^T$, respectively, to obtain $R_A=A(I_A,:)$ and $R_B=B(I_B,:)$ with $\lvert I_A\rvert=\lvert I_B\rvert =l$.	\State Compute $M_A=C_A^{\dagger}AR_A^{\dagger}$ and $M_B=C_B^{\dagger}BR_B^{\dagger}$.	
	\end{algorithmic}
\end{algorithm}
\begin{remark}\label{remark10}
In Algorithm \ref{alg3}, two passes through $(A,B)$ are required. One is used to obtain the sketches $X$, $Y_1$ and $Y_2$, and further to find the index sets $I_A, I_B$ and $J$. The other is used for
retrieving the factors of CUR decomposition. So, when only the column and row index sets are required, Algorithm \ref{alg3} only needs one pass and hence can be adapted to the streaming setting.  In contrast, implementing Algorithm \ref{alg2} needs to access $(A,B)$ at least four times. 
\end{remark}


Note that $Y_1^T\Pi_{1r} = Q_1T_{11}$ and $ Y_2^T\Pi_{2r} = Q_2T_{21}$ are invertible. Then, we can define two oblique projectors: 
$P_{Y_1,N_4}=Y_1(\Pi_{1r}^TY_1)^{-1}\Pi_{1r}^T$ and $P_{Y_2,N_5}=Y_2(\Pi_{2r}^TY_2)^{-1}\Pi_{2r}^T$.
With them, we have the following results. 
\begin{theorem} \label{theorem2}
	Let $C_A\in \mathbb{R}^{m \times l}, C_B\in \mathbb{R}^{d \times l}, R_A\in \mathbb{R}^{l \times n}, R_B\in \mathbb{R}^{l \times n}, M_A\in \mathbb{R}^{l \times l}$, and $ M_B\in \mathbb{R}^{l \times l}$ be obtained by Algorithm \ref{alg3}. Then 
	\begin{align}
		\Vert A-C_AM_AR_A \Vert &\leq 2\Vert I_n-P_{\Pi_C,N_1} \Vert \Vert A(I_n-X^{\dagger}X) \Vert+ \Vert I_m-P_{Y_1,N_4} \Vert \Vert (I_m-Y_1Y_1^{\dagger})A \Vert,  \label{single A}\\ 
		\Vert B-C_BM_BR_B \Vert &\leq 2\Vert I_n-P_{\Pi_C,N_1} \Vert \Vert B(I_n-X^{\dagger}X) \Vert+ \Vert I_d-P_{Y_2,N_5} \Vert \Vert (I_d-Y_2Y_2^{\dagger})B \Vert. \label{single B}
	\end{align}
\end{theorem}
\begin{proof}
	We just consider the error bound \eqref{single A} for $A$. 
 As the proof in \cite[Lemma 3.1]{voronin2017efficient}, we have
	\begin{align*}
		\Vert A-C_AM_AR_A \Vert \leq \Vert A-C_AV^T \Vert + \Vert A-W_AR_A \Vert.
	\end{align*}
Since the bound of the first term in the right-hand side of the above inequality has been deduced in Theorem \ref{theorem1}, we next study the bound of the second term. 

On one hand, we have 
\begin{align}
    \Vert A-W_AR_A \Vert & = \Vert A-C_AC_A^{\dagger}AR_A^{\dagger}R_A \Vert \notag \\
    & = \Vert A-C_AC_A^{\dagger}A+C_AC_A^{\dagger}A-C_AC_A^{\dagger}AR_A^{\dagger}R_A \Vert \notag \\
    & \leq \Vert A-C_AC_A^{\dagger}A \Vert+\Vert C_AC_A^{\dagger} \Vert \cdot \Vert A-AR_A^{\dagger}R_A \Vert \notag \\
    & =\Vert A-C_AC_A^{\dagger}A \Vert+\Vert A-AR_A^{\dagger}R_A \Vert,
    \label{AWR2}
\end{align}
where the last equality holds because $C_AC_A^{\dagger}$ is an orthogonal projector. On the other hand, considering that $R_A=\Pi_{1r}^TA$ is full row rank, using the property of Moore-Penrose inverse, we obtain
\begin{align*}
 AR_A^{\dagger}R_A=AR_A^T(R_AR_A^T)^{-1}\Pi_{1r}^TA:=P_{AR_A^T,N_6}A,
\end{align*}
where 
$P_{AR_A^T,N_6}=AR_A^T(R_AR_A^T)^{-1}\Pi_{1r}^T$.
It is easy to find that $P_{Y_1,N_4}P_{AR_A^T,N_6}=P_{Y_1,N_4}$.
Hence, 
\begin{align}
 \Vert A-AR_A^{\dagger}R_A \Vert   &= \Vert  A-P_{AR_A^T,N_6}A \Vert \notag =\Vert  (I_m-P_{Y_1,N_4})(I_m-P_{AR_A^T,N_6})A \Vert \\ &=\Vert  (I_m-P_{Y_1,N_4})A(I_n-R_A^{\dagger}R_A) \Vert \notag \\ & \leq \Vert  (I_m-P_{Y_1,N_4})A\Vert \Vert I_n -R_A^{\dagger}R_A \Vert  = \Vert  (I_m-P_{Y_1,N_4})A\Vert \notag \\ & \leq \Vert I_m-P_{Y_1,N_4} \Vert \Vert (I_m-Y_1Y_1^{\dagger})A \Vert, \label{AWR1}
\end{align}
where the last equality holds because $I_n-R_A^{\dagger}R_A$ is an orthogonal projector and the last inequality relies on the following result
\begin{align*}
    P_{Y_1,N_4}Y_1-Y_1=0.
\end{align*}
Thus, (\ref{single A}) can be obtained by combining  (\ref{AWR2}) and (\ref{AWR1}) with (\ref{column A}). 
\end{proof}

Now, we present the expectation error bounds for Algorithm \ref{alg3}.   For brevity, we only consider the case for spectral norm. 
\begin{theorem} \label{theorem-err2}
	With the same setting as Theorem \ref{theorem2}, let $k$ be the target rank and $p$ be the oversampling parameter, and let 
	{\small\begin{align*}
		\alpha = \sqrt{1+(n-k-p)4^{k+p-1}}  
		 \left[ \left( 1+\sqrt{\frac{k}{p-1}} \right) \sigma_{k+1}\begin{bmatrix}
			A \\ B
		\end{bmatrix} +\frac{{\rm e}\sqrt{k+p}}{p}\left( \sum_{j>k}^{}\sigma_j^2\begin{bmatrix}
			A \\ B
		\end{bmatrix} \right)^{1/2} \right].
	\end{align*}}
	Then 
	\begin{align*}
		\mathbb{E}\Vert A-C_AM_AR_A \Vert &\leq 2\alpha + \sqrt{1+(m-k-p)4^{k+p-1}}  \\
  &\times\left[ \left( 1+\sqrt{\frac{k}{p-1}} \right) \sigma_{k+1}(A) +\frac{{\rm e}\sqrt{k+p}}{p}\left( \sum_{j>k}^{}\sigma_j^2(A) \right)^{1/2} \right], \\
		\mathbb{E}\Vert B-C_BM_BR_B \Vert &\leq 2\alpha + \sqrt{1+(d-k-p)4^{k+p-1}}  \\
  &\times\left[ \left( 1+\sqrt{\frac{k}{p-1}} \right) \sigma_{k+1}(B) +\frac{{\rm e}\sqrt{k+p}}{p}\left( \sum_{j>k}^{}\sigma_j^2(B) \right)^{1/2} \right].
	\end{align*}
\end{theorem}
\begin{proof}	
	By virtue of  Lemma 3.2 in \cite{dong2021simpler}, we have 
	\begin{eqnarray*}
  &\Vert I_n-P_{\Pi_C,N_1} \Vert \leq \sqrt{1+(n-k-p)4^{k+p-1}}, \Vert I_m-P_{Y_1,N_4} \Vert  \leq \sqrt{1+(m-k-p)4^{k+p-1}}, \\ &\Vert I_d-P_{Y_2,N_5} \Vert  \leq \sqrt{1+(d-k-p)4^{k+p-1}}.
	\end{eqnarray*}
	Thus, combining the fact 
	\begin{align*}
		&{\rm max}\{\Vert A(I_n-X^{\dagger}X) \Vert, \Vert B(I_n-X^{\dagger}X) \Vert\} \leq  \left \Vert \begin{bmatrix}
			A \\B
		\end{bmatrix}(I_n-X^{\dagger}X) \right \Vert
	\end{align*}
	with Theorem \ref{theorem2} and Lemma \ref{lem-expectation}, the proof is completed.
\end{proof}

\subsection{GSVD-based error analyses} \label{GSVD based error analyses}
The error bounds in Theorems \ref{theorem-err1} and \ref{theorem-err2} are controlled partly by the  singular values of $\begin{bmatrix}
	A^T & B^T
\end{bmatrix}^T$. 
In this subsection, we provide the alternatives error analyses by the GSVD of the matrix pair $(A,B)$ in (\ref{gsvd}).

Let the QR decomposition of $Y^T$ be $Y^T=QR$. 
It is clear that $R\in \mathbb{R}^{n\times n}$ is nonsingular since $Y^T\in \mathbb{R}^{n\times n}$ is nonsingular. Thus, from the proof of Theorem \ref{theorem1} and the fact $(I_m-C_AC_A^\dagger )A=A-C_AM_AR_A $, the following inequality holds  
\begin{align}
	\Vert A-C_AM_AR_A  \Vert  &\leq \Vert A(I_n-X^\dagger X) (I_n-P_{\Pi_C,N_1})\Vert \notag \\
	&= \Vert A(I_n-X^\dagger X)R^{-1} R(I_n-P_{\Pi_C,N_1})\Vert \notag \\
	&\leq  \Vert A(I_n-X^\dagger X)R^{-1}\Vert \Vert R(I_n-P_{\Pi_C,N_1})\Vert.  \label{apen5}
\end{align}
Similarly, we also have $\Vert B-C_BM_BR_B  \Vert \leq  \Vert B(I_n-X^\dagger X)R^{-1}\Vert \Vert R(I_n-P_{\Pi_C,N_1})\Vert$. In the following, we bound $ \Vert A(I_n-X^\dagger X)R^{-1}\Vert$ and $\Vert B(I_n-X^\dagger X)R^{-1}\Vert$. 


Denote
\begin{align*}
	\hat{A}=AR^{-1},\ \hat{B}=BR^{-1},\textrm{ and } \hat{X}=XR^{-1},
\end{align*}
and suppose that 
${\rm range}(R^{-1}(R^{-1})^TX^T)\subset {\rm range}(X^T)$. 
Hence, combining with the equation
\begin{align*}
	\begin{bmatrix}
		A\\B
	\end{bmatrix}=\begin{bmatrix}
		U & \\ & \tilde{V}
	\end{bmatrix}\begin{bmatrix}
		\Sigma_1 \\ \Sigma_2 
	\end{bmatrix}Y^T,
\end{align*}
we have 
\begin{align*}
	\hat{X}=\Omega \begin{bmatrix}
		A \\ B
	\end{bmatrix}R^{-1}=\Omega \begin{bmatrix}
		\hat{A} \\ \hat{B}
	\end{bmatrix}.
\end{align*}
To continue the GSVD based error analyses, the following lemma is necessary. 
\begin{lemma} \cite[Ex.22]{Ben2003generalized} \label{fanxu}
    Given matrices $D_1\in \mathbb{R}^{t_1\times t_2}$ and $D_2\in \mathbb{R}^{t_2\times t_3}$, then 
    \begin{align*}   (D_1D_2)^{\dagger}=D_2^{\dagger}D_1^{\dagger}
    \end{align*} if and only if 
    \begin{align*}
        {\rm range}(D_1^TD_1D_2)\subset {\rm range}(D_2), \ {\rm range}(D_2D_2^TD_1^T)\subset {\rm range}(D_1^T).
    \end{align*}
\end{lemma}
By means of Lemma \ref{fanxu}, and considering the assumption ${\rm range}(R^{-1}(R^{-1})^TX^T)\subset {\rm range}(X^T)$ and the fact ${\rm range}(X^TXR^{-1})\subset {\rm range}(R^T)$, we have $(XR^{-1})^{\dagger}=RX^{\dagger}$. 
Hence, 
\begin{align}
	\Vert \hat{A}  (I_n-\hat{X}^{\dagger}\hat{X}) \Vert &=\Vert AR^{-1} (I_n-(XR^{-1})^{\dagger}(XR^{-1})) \Vert=\Vert AR^{-1} (I_n-RX^{\dagger}XR^{-1}) \Vert \notag \\
	&=\Vert A(R^{-1}-X^{\dagger}XR^{-1}) \Vert = \Vert A(I_n-X^{\dagger}X)R^{-1} \Vert. \label{apen7}
\end{align}
For brevity, here we just discuss the case of $A$ and that of $B$ can be derived similarly. From \eqref{apen7}, it suffices to bound $\Vert \hat{A}  (I_n-\hat{X}^{\dagger}\hat{X}) \Vert$. 
In the following, we tackle the problem using the idea 
from 
\cite{halko2011finding}. 

Observe that the matrix $\hat{X}$ can be expressed as follows by partitioning  $U= \begin{bmatrix}
		U_1 , U_2 
	\end{bmatrix}, \tilde{V}=\begin{bmatrix}
		V_1 , V_2
	\end{bmatrix}$, 
 \begin{align}
 \Sigma_1=\begin{bmatrix}
		\Sigma_{11} & \\ &   \Sigma_{12}
	\end{bmatrix}, \textrm{ and }
	\Sigma_2=\begin{bmatrix}
		\Sigma_{21} & \\ &   \Sigma_{22}
	\end{bmatrix} \label{apen6}
 \end{align}
    with $\Sigma_{11}\in \mathbb{R}^{k\times k}, \Sigma_{12}\in \mathbb{R}^{(m-k)\times (n-k)}, \Sigma_{21}\in \mathbb{R}^{k\times k}$ and $\Sigma_{22}\in \mathbb{R}^{(d-k)\times (n-k)}$: 
\begin{align*}
	\hat{X}&=\Omega \begin{bmatrix}
		U_1 & U_2 & & \\
		& & V_1 & V_2
	\end{bmatrix}\begin{bmatrix}
		\Sigma_{11} & \\ & \Sigma_{12} \\ 
		\Sigma_{21} & \\ & \Sigma_{22} 
	\end{bmatrix}Q  =\Omega \begin{bmatrix}
		\hat{U_1}, \hat{U_2}, \hat{V_1}, \hat{V_2}
	\end{bmatrix}\begin{bmatrix}
		\Sigma_{11} & \\ & \Sigma_{12} \\ 
		\Sigma_{21} & \\ & \Sigma_{22} 
	\end{bmatrix}Q \notag \\ 
	&=\Omega \begin{bmatrix}
		\hat{U_1}\Sigma_{11}+\hat{V_1}\Sigma_{21}, \hat{U_2}\Sigma_{12}+\hat{V_2}\Sigma_{22}  
	\end{bmatrix}Q ,    
\end{align*} 
where 
\begin{align*}
	\hat{U_1}=\begin{bmatrix}
		U_1 \\ 0
	\end{bmatrix}, 	\hat{U_2}=\begin{bmatrix}
		U_2 \\ 0
	\end{bmatrix}, \hat{V_1}=\begin{bmatrix}
		0\\ V_1 
	\end{bmatrix}, \textrm{ and }\hat{V_2}=\begin{bmatrix}
		0\\ V_2
	\end{bmatrix}. 
\end{align*}
Considering that the case we focus on is that both $A$ and $ B$ are low rank, $\Sigma_{11}$ and $\Sigma_{22}$ may be nearly zero. 
Then 
\begin{align*}
	\hat{X} \approx \Omega \begin{bmatrix}
		\hat{V_1}\Sigma_{21}, \hat{U_2}\Sigma_{12}
	\end{bmatrix}Q=\begin{bmatrix}
		\Omega\hat{V_1}\Sigma_{21}, \Omega\hat{U_2}\Sigma_{12}
	\end{bmatrix}Q=\begin{bmatrix}
		\Gamma_1\Sigma_{21}, \Gamma_2\Sigma_{12}
	\end{bmatrix}Q,
\end{align*}
where $\Gamma_1=\Omega\hat{V_1}\in \mathbb{R}^{(k+p)\times k}$ and $\Gamma_2=\Omega\hat{U_2}$. 
Further, $\hat{X}Q^T\approx\begin{bmatrix}
	\Gamma_1\Sigma_{21}, \Gamma_2\Sigma_{12}
\end{bmatrix}$. Let $\tilde{X}=\hat{X}Q^T$, note that $\Sigma_{21}$ is invertible, and suppose that $\Gamma_1$ is full column rank. 
Then
\begin{align*}
	\hat{Z}:=\Sigma_{21}^{-1}\Gamma_1^{\dagger}\tilde{X}\approx\begin{bmatrix}
		I , \hat{F}
	\end{bmatrix},
\end{align*}
where $\hat{F}=\Sigma_{21}^{-1}\Gamma_1^{\dagger}\Gamma_2\Sigma_{12}$. It is obvious that ${\rm range}(\hat{Z}^T)\subset {\rm range}(\tilde{X}^T)$ and $\hat{Z}$ is full row rank. 

To continue to bound $\Vert \hat{A}  (I_n-\hat{X}^{\dagger}\hat{X}) \Vert$, 
some notations are necessary. Let
\begin{align*}
	&\tilde{A}=U\Sigma_1=\hat{A}Q^T,\ \tilde{B}=\tilde{V}\Sigma_2=\hat{B}Q^T,\ P_{\hat{Z}}=\hat{Z}^T(\hat{Z}\hat{Z}^T)^{-1}\hat{Z}, \\
    &P_{\tilde{X}}=\tilde{X}^T(\tilde{X}\tilde{X}^T)^{\dagger}\tilde{X}=Q\hat{X}^T(\hat{X}\hat{X}^T)^{\dagger}\hat{X}Q^T=Q\hat{X}^{\dagger}\hat{X}Q^T,
\end{align*}
where the last equality holds since $(\hat{X}\hat{X}^T)^{\dagger}=(\hat{X}^T)^{\dagger}\hat{X}^{\dagger}$. 
Note that $P_{\hat{Z}}$ and $P_{\tilde{X}}$ are orthogonal projectors. Thus,
combining ${\rm range}(\hat{Z}^T)\subset {\rm range}(\tilde{X}^T)$ with Proposition 8.5 in \cite{halko2011finding}, we have
\begin{align*}
	\Vert \hat{A}  (I_n-\hat{X}^{\dagger}\hat{X}) \Vert =
	\Vert \tilde{A} Q (I_n-\hat{X}^{\dagger}\hat{X})Q^T \Vert=\Vert \tilde{A}(I_n-P_{\tilde{X}}) \Vert 
	\leq \Vert \tilde{A}(I_n-P_{\hat{Z}}) \Vert.
\end{align*}
Further, we can obtain
\begin{align}
	\Vert \hat{A}  (I_n-\hat{X}^{\dagger}\hat{X}) \Vert^2 &\leq \Vert \tilde{A}(I_n-P_{\hat{Z}}) \Vert^2 =\Vert U\Sigma_1 (I_n-P_{\hat{Z}})\Sigma_1^TU^T \Vert \notag \\
	&=\Vert \Sigma_1 (I_n-P_{\hat{Z}})\Sigma_1^T \Vert,  \label{apen1}
\end{align}
where the first equality holds since $I_n-P_{\hat{Z}}$ is an orthogonal projector. As done in the proof of \cite[Theorem 9.1]{halko2011finding}, by the similar algebraic computation, the following holds,  


\begin{align*}
	\Sigma_1(I_n-P_{\hat{Z}})\Sigma_1^T \preccurlyeq \begin{bmatrix}
		\Sigma_{11}\hat{F}\hat{F}^T\Sigma_{11}^T & -\Sigma_{11}\left( I_k+\hat{F}\hat{F}^T \right)^{-1}\hat{F}\Sigma_{12}^T \\
		-\Sigma_{12}\hat{F}^T\left( I_k+\hat{F}\hat{F}^T \right)^{-1}\Sigma_{11}^T & \Sigma_{12}\Sigma_{12}^T
	\end{bmatrix}.
\end{align*}
Note that the matrices on the two sides of the above inequality are positive semi-definite. Hence,
Proposition 8.3 in \cite{halko2011finding} results in 
\begin{align*}
	\Vert \Sigma_1(I_n-P_{\hat{Z}})\Sigma_1^T \Vert \leq \Vert \Sigma_{11}\hat{F}\hat{F}^T\Sigma_{11}^T  \Vert + \Vert \Sigma_{12}\Sigma_{12}^T \Vert =
	\Vert \Sigma_{11}\hat{F} \Vert^2+\Vert \Sigma_{12}\Vert^2.
\end{align*}
Thus, combining (\ref{apen1}) with the above inequality implies 
\begin{align}
	\Vert \hat{A}(I_n-\hat{X}^{\dagger}\hat{X}) \Vert^2 \leq \Vert \Sigma_{11}\hat{F} \Vert^2+\Vert \Sigma_{12}\Vert^2.  \label{apinfang}
\end{align}
With the fact that for all $a,b\geq 0$, $c^2\leq a^2+b^2$ results in $c\leq a+b$, we have 
\begin{align*}
	\Vert \hat{A}(I_n-\hat{X}^{\dagger}\hat{X}) \Vert \leq  \Vert \Sigma_{11}\hat{F} \Vert +\Vert \Sigma_{12}\Vert \leq \Vert \Sigma_{11}\Sigma_{21}^{-1}\Vert \Vert \Gamma_1^{\dagger}\Gamma_2\Sigma_{12}  \Vert +\Vert \Sigma_{12} \Vert ,  
\end{align*}
where $\hat{F}=\Sigma_{21}^{-1}\Gamma_1^{\dagger}\Gamma_2\Sigma_{12}$. 

Similarly, for the case of $B$, we have
\begin{align}
	\Vert \hat{B}  (I_n-\hat{X}^{\dagger}\hat{X}) \Vert^2 &\leq \Vert \Sigma_2 (I_n-P_{\hat{Z}})\Sigma_2^T \Vert,\nonumber  \\
	\Sigma_2 (I_n-P_{\hat{Z}})\Sigma_2^T &\preccurlyeq \begin{bmatrix}
		\Sigma_{21}\hat{F}\hat{F}^T\Sigma_{21}^T & -\Sigma_{21}\left( I_k+\hat{F}\hat{F}^T \right)^{-1}\hat{F}\Sigma_{22}^T \\
		-\Sigma_{22}\hat{F}^T\left( I_k+\hat{F}\hat{F}^T \right)^{-1}\Sigma_{21}^T & \Sigma_{22}\Sigma_{22}^T
	\end{bmatrix},\nonumber\\
	\Vert \hat{B}  (I_n-\hat{X}^{\dagger}\hat{X}) \Vert^2 &\leq \Vert \Sigma_{21}\hat{F} \Vert^2 + \Vert \Sigma_{22} \Vert^2 \leq \Vert \Sigma_{21}\Sigma_{21}^{-1}\Gamma_1^{\dagger}\Gamma_2\Sigma_{12} \Vert^2 + \Vert \Sigma_{22} \Vert^2 \notag \\
	& = \Vert \Gamma_1^{\dagger}\Gamma_2\Sigma_{12} \Vert^2 + \Vert \Sigma_{22} \Vert^2, \label{bpinfang}
\end{align}
and 
\begin{align*}
	\Vert \hat{B}  (I_n-\hat{X}^{\dagger}\hat{X}) \Vert  \leq \Vert \Gamma_1^{\dagger}\Gamma_2\Sigma_{12} \Vert + \Vert \Sigma_{22} \Vert. 
\end{align*}

Therefore,
\begin{align}
\Vert A(I_n-X^\dagger X)R^{-1}\Vert &\leq \Vert \Sigma_{11}\Sigma_{21}^{-1}\Vert \Vert \Gamma_1^{\dagger}\Gamma_2\Sigma_{12}  \Vert +\Vert \Sigma_{12} \Vert, \label{apen222} \\ 
\Vert B(I_n-X^\dagger X)R^{-1}\Vert  &\leq \Vert \Gamma_1^{\dagger}\Gamma_2\Sigma_{12} \Vert + \Vert \Sigma_{22} \Vert. \label{apen322}
\end{align}
Furthermore, it is easy to see that (\ref{apinfang}) and (\ref{bpinfang}) also hold for Frobenius norm.

Note that, in the above discussions, two necessary assumptions are used. We list them  uniformly as follows. 
\begin{assumption} \label{assum}
	(1) ${\rm range}(R^{-1}(R^{-1})^TX^T)\subset {\rm range}(X^T)$; 
	(2) $\Gamma_1=\Omega \hat{V}_1$ 
 is full column rank.
\end{assumption}
\begin{remark}
    The last assumption is satisfied in high probability when $\Omega$ is a Gaussian matrix; see e.g., \cite{halko2011finding}.
\end{remark}

In the following, we present the expectation error bounds of Algorithm \ref{alg2}. Two necessary lemmas are listed firstly as follows. 
\begin{lemma}\cite[Proposition 10.1]{halko2011finding} \label{Proposition 10.1}
	For the given matrices $S$ and $N$, draw a standard Gaussian matrix $\Gamma$. Then
	\begin{align*}
		(\mathbb{E}\Vert S\Gamma N \Vert_{F}^2)^{\frac{1}{2}}=\Vert S\Vert_{F}\Vert N\Vert_{F},\ \mathbb{E}\Vert S\Gamma N \Vert\leq \Vert S \Vert \Vert N \Vert_{F}+\Vert S \Vert_{F} \Vert N \Vert.
	\end{align*}
\end{lemma}

\begin{lemma}\cite[Proposition 10.2]{halko2011finding} \label{Proposition 10.2}
	Let $\Gamma \in \mathbb{R}^{(k+p)\times k}$ be a standard Gaussian matrix. Then 
	\begin{align*}
		(\mathbb{E}\Vert \Gamma^{\dagger} \Vert_{F}^2)^{\frac{1}{2}}=\sqrt{\frac{k}{p-1}},\ \mathbb{E}\Vert \Gamma^{\dagger} \Vert \leq \frac{{\rm e}\sqrt{k+p}}{p}.
	\end{align*}
\end{lemma}

\begin{theorem} \label{theorem-err3}
With the same setting as Theorem \ref{theorem-err1} and Assumption \ref{assum},
 let 
 \begin{align*}
     \eta =\Vert Y \Vert \sqrt{1+(n-k-p)4^{k+p-1}} 
 \end{align*}
 with $Y$ being defined in (\ref{gsvd}).
 Then 
	\begin{align*}
		&\mathbb{E}	\Vert A-C_AM_AR_A \Vert \leq \eta  \left[ \Vert \Sigma_{11}\Sigma_{21}^{-1}\Vert \left( \frac{{\rm e}\sqrt{k+p}}{p}\Vert \Sigma_{12} \Vert_F + \sqrt{\frac{k}{p-1}}\Vert \Sigma_{12}\Vert \right) + \Vert\Sigma_{12}\Vert \right], \\
		&\mathbb{E}	\Vert B-C_BM_BR_B \Vert \leq  \eta \left[ \left( \frac{{\rm e}\sqrt{k+p}}{p}\Vert \Sigma_{12} \Vert_F + \sqrt{\frac{k}{p-1}}\Vert \Sigma_{12}\Vert \right)+ \Vert\Sigma_{22}\Vert \right],\\
		&\mathbb{E}	\Vert A-C_AM_AR_A \Vert_F \leq  \eta  \left[ \sqrt{1+\frac{k}{p-1}\Vert \Sigma_{11}\Sigma_{21}^{-1} \Vert}\Vert \Sigma_{12} \Vert_F  \right], \\
		&\mathbb{E}	\Vert B-C_BM_BR_B \Vert_F \leq  \eta \left[ \frac{k}{p-1}\Vert \Sigma_{12} \Vert_F^2+\Vert \Sigma_{22} \Vert_F^2 \right]^{1/2},
	\end{align*}
 where $\Sigma_{11}, \Sigma_{12}, \Sigma_{21}$, and $ \Sigma_{22}$ are defined as in  (\ref{apen6}).
\end{theorem}

\begin{proof}
	Note that $Y^T=QR$, then $\Vert Y^T \Vert=\Vert Y \Vert=\Vert R \Vert$. Thus, noting  (\ref{apen5}), 
 we have 
	\begin{align}
		&\Vert A-C_AM_AR_A \Vert \leq \Vert Y \Vert \Vert I_n-P_{\Pi_C,N_1} \Vert \Vert A(I_n-X^\dagger X)R^{-1} \Vert, \label{apen8} \\
		&\Vert A-C_AM_AR_A \Vert_F \leq \Vert Y \Vert \Vert I_n-P_{\Pi_C,N_1} \Vert \Vert A(I_n-X^\dagger X)R^{-1}\Vert_F. \label{apen9}
	\end{align} 
	Similarly, for the case of $B$, 
	\begin{align}
		&\Vert B-C_BM_BR_B \Vert \leq \Vert Y \Vert \Vert I_n-P_{\Pi_C,N_1} \Vert \Vert B(I_n-X^\dagger X)R^{-1} \Vert, \label{apen10} \\
		&\Vert B-C_BM_BR_B \Vert_F \leq \Vert Y \Vert \Vert I_n-P_{\Pi_C,N_1} \Vert \Vert B(I_n-X^\dagger X)R^{-1} \Vert_F. \label{apen11}
	\end{align}
	By virtue of Lemma 3.2 in \cite{dong2021simpler}, we get
	\begin{align}
		\Vert I_n-P_{\Pi_C,N_1} \Vert \leq \sqrt{1+(n-k-p)4^{k+p-1}}. \label{apen12}
	\end{align}
	Next, it suffices to bound the expectations of $\Vert A(I-X^\dagger X)R^{-1} \Vert, \Vert B(I-X^\dagger X)R^{-1} \Vert$, $\Vert A(I-X^\dagger X)R^{-1} \Vert_F$, and $ \Vert B(I-X^\dagger X)R^{-1} \Vert_F$. 
	Observing (\ref{apen222}) and (\ref{apen322}), we need to take expectation on $\Vert  \Gamma_1^{\dagger}\Gamma_2\Sigma_{12}\Vert$ and $\Vert  \Gamma_1^{\dagger}\Gamma_2\Sigma_{12}\Vert_F^2$. On one hand,
	\begin{align}
		\mathbb{E}\Vert  \Gamma_1^{\dagger}\Gamma_2\Sigma_{12}\Vert &= \mathbb{E}\left( \mathbb{E}\left[\Vert  \Gamma_1^{\dagger}\Gamma_2\Sigma_{12}\Vert \Big\vert \Gamma_1 \right] \right) \leq \mathbb{E} \left( \Vert  \Gamma_1^{\dagger}\Vert \Vert\Sigma_{12}\Vert_F+\Vert  \Gamma_1^{\dagger}\Vert_F \Vert\Sigma_{12}\Vert \right) \notag \\
		&\leq \Vert\Sigma_{12}\Vert_F \mathbb{E}\Vert  \Gamma_1^{\dagger}\Vert +\Vert\Sigma_{12}\Vert \left(\mathbb{E}\Vert  \Gamma_1^{\dagger}\Vert_F^2 \right)^{1/2}  \notag \\
		&\leq  \frac{{\rm e}\sqrt{k+p}}{p}\Vert \Sigma_{12} \Vert_F + \sqrt{\frac{k}{p-1}}\Vert \Sigma_{12}\Vert, \label{apen13}
	\end{align}
	where the first equality relies on the fact that $\Gamma_1\in \mathbb{R}^{(k+p) \times k}$ and $ \Gamma_2$ are independent, the first inequality relies on Lemma \ref{Proposition 10.1}, the second inequality relies on H$\ddot{\rm o}$lder’s inequality, and the third inequality relies on Lemma \ref{Proposition 10.2}. On the other hand,
	\begin{align}
		\mathbb{E}\Vert  \Gamma_1^{\dagger}\Gamma_2\Sigma_{12}\Vert_F^2 &= \mathbb{E}\left( \mathbb{E}\left[\Vert  \Gamma_1^{\dagger}\Gamma_2\Sigma_{12}\Vert_F^2 \Big\vert  \Gamma_1 \right] \right) =\mathbb{E}\left( \Vert  \Gamma_1^{\dagger}\Vert_F^2 \Vert  \Sigma_{12}\Vert_F^2 \right) \notag \\
  &= \Vert  \Sigma_{12}\Vert_F^2 \mathbb{E}\Vert  \Gamma_1^{\dagger}\Vert_F^2 =\frac{k}{p-1}\Vert  \Sigma_{12}\Vert_F^2,  \label{apen14}
	\end{align}
	where the second equality relies on Lemma \ref{Proposition 10.1} and the last equality relies on Lemma \ref{Proposition 10.2}.
	Thus, taking the expectation of both sides of (\ref{apen222}) and (\ref{apen322}) implies 
	\begin{align}
		&\mathbb{E}\Vert A(I_n-X^\dagger X)R^{-1} \Vert \leq  \Vert \Sigma_{11}\Sigma_{21}^{-1} \Vert \cdot \mathbb{E}\Vert \Gamma_1^{\dagger}\Gamma_2\Sigma_{12} \Vert + \Vert \Sigma_{12} \Vert, \label{apen15} \\
		&\mathbb{E}\Vert B(I_n-X^\dagger X)R^{-1} \Vert \leq \mathbb{E}\Vert \Gamma_1^{\dagger}\Gamma_2\Sigma_{12} \Vert + \Vert \Sigma_{22} \Vert. \label{apen16}
	\end{align}
	Similarly, take the expectation of both sides of (\ref{apinfang}) and (\ref{bpinfang}) leading to
	\begin{align}
		\mathbb{E}\Vert A(I_n-X^\dagger X)R^{-1} \Vert_F &\leq \left( \mathbb{E}\Vert A(I_n-X^\dagger X)R^{-1} \Vert_F^2 \right)^{1/2} \notag \\
  &\leq \left( \Vert \Sigma_{11}\Sigma_{21}^{-1} \Vert^2 \cdot \mathbb{E}\Vert \Gamma_1^{\dagger}\Gamma_2\Sigma_{12} \Vert_F^2 +\Vert \Sigma_{12} \Vert_F^2 \right)^{1/2} , \label{apen17} \\
		 \mathbb{E}\Vert B(I_n-X^\dagger X)R^{-1} \Vert_F &\leq \left( \mathbb{E}\Vert B(I_n-X^\dagger X)R^{-1} \Vert_F^2 \right)^{1/2} \notag \\
   &\leq \left( \mathbb{E}\Vert \Gamma_1^{\dagger}\Gamma_2\Sigma_{12} \Vert_F^2 +\Vert \Sigma_{22} \Vert_F^2 \right)^{1/2}. \label{apen18}
	\end{align}
	Finally, taking the expectation of both sides of (\ref{apen8}), (\ref{apen9}), (\ref{apen10}), and (\ref{apen11}), and then combining them with (\ref{apen12}), (\ref{apen13}), (\ref{apen14}), (\ref{apen15}), (\ref{apen16}), (\ref{apen17}), and (\ref{apen18})  complete the proof.
\end{proof}

Similarly, combining Theorem \ref{theorem2} with Lemma \ref{lem-expectation}, we can obtain the error analysis for 
Algorithm \ref{alg3}.
\begin{corollary}\label{COR1}
	With the same setting  as Theorems \ref{theorem2} and \ref{theorem-err3}, we have  
	{\small\begin{align*}
		&\mathbb{E}	\Vert A-C_AM_AR_A \Vert \leq 2\eta  \left[ \Vert \Sigma_{11}\Sigma_{21}^{-1}\Vert \left( \frac{{\rm e}\sqrt{k+p}}{p}\Vert \Sigma_{12} \Vert_F + \sqrt{\frac{k}{p-1}}\Vert \Sigma_{12}\Vert \right) + \Vert\Sigma_{12}\Vert \right]  \\
		&+\sqrt{1+(m-k-p)4^{k+p-1}} \left[ \left( 1+\sqrt{\frac{k}{p-1}} \right) \sigma_{k+1}(A) +\frac{{\rm e}\sqrt{k+p}}{p}\left( \sum_{j>k}^{}\sigma_j^2(A) \right)^{1/2} \right], \\
		&\mathbb{E}	\Vert B-C_BM_BR_B \Vert \leq 2\eta  \left[ \left( \frac{{\rm e}\sqrt{k+p}}{p}\Vert \Sigma_{12} \Vert_F + \sqrt{\frac{k}{p-1}}\Vert \Sigma_{12}\Vert \right)+ \Vert\Sigma_{22}\Vert \right] + \\
		&\sqrt{1+(d-k-p)4^{k+p-1}} \left[ \left( 1+\sqrt{\frac{k}{p-1}} \right) \sigma_{k+1}(B) +\frac{{\rm e}\sqrt{k+p}}{p}\left( \sum_{j>k}^{}\sigma_j^2(B) \right)^{1/2} \right].
	\end{align*}}
\end{corollary}

\section{Randomized algorithms for CUR decomposition of matrix triplet}
\label{sec-triplets}
We first consider a randomized algorithm based on randomized CPQR. 
Let $A\in \mathbb{R}^{m \times n}$, $B\in \mathbb{R}^{m \times t}$, $G\in \mathbb{R}^{d \times n}$, ${\rm rank}(A) > k, {\rm rank}(B) > k$, and ${\rm rank}(G) > k$, 
and let $\Omega_2 \in \mathbb{R}^{l \times (m+d)}$ and $\Omega_3 \in \mathbb{R}^{l \times (n+t)}$ be Gaussian matrices. Define 
\begin{align*}
	X_1=\Omega_2\begin{bmatrix}
		A \\ G
	\end{bmatrix},\ Y_3=\begin{bmatrix}
	A, B
\end{bmatrix}\Omega_3^T.
\end{align*}
By virtue of the CPQRs of $X_1$ and $Y_3^T$, we have
 \begin{align*}
 	\underset{l\times n}{X_1}\underset{n\times n}{\Pi_3}=X_1\begin{bmatrix}
 		\underset{n\times l}{\Pi_{31}},  \underset{n\times (n-l)}{\Pi_{31}^c}
 	\end{bmatrix}=\underset{l\times l}{Q_3}\begin{bmatrix}
 	\underset{l\times l}{T_{31}}, \underset{l\times (n-l)}{T_{32}}
 \end{bmatrix}, \\
\underset{l\times m}{Y_3^T}\underset{m\times m}{\Pi_4}=Y_3^T\begin{bmatrix}
	\underset{m\times l}{\Pi_{41}},  \underset{m\times (m-l)}{\Pi_{41}^c}
\end{bmatrix}=\underset{l\times l}{Q_4}\begin{bmatrix}
	\underset{l\times l}{T_{41}}, \underset{l\times (m-l)}{T_{42}}
\end{bmatrix},
 \end{align*}
where $\Pi_3$ and $ \Pi_4$ are permutation matrices, $Q_3$ and $ Q_4$ are orthogonal, and $T_{31}$ and $ T_{41}$ are invertible and upper triangular. With this step, we can select the columns of $A$ and $G$, and the rows of $A$ and $B$. So, it suffices to consider how to select the rows of $G$ and the columns of $B$. To this end, 
similar to the discussion on the row IDs of $C_A$ and $C_B$ in Section \ref{sec-pairs}, we perform the exact IDs for the columns of $G$ and the rows of $B$ to obtain them. 
The specific algorithm 
is concluded in Algorithm \ref{alg4}.

\begin{algorithm}[ht] 
	\caption{Randomized algorithm for the CUR decomposition of $(A,B,G)$} 
	\label{alg4}
	\hspace*{0.02in} {\bf Input:} 
	$A\in \mathbb{R}^{m \times n}$, $B\in \mathbb{R}^{m \times t}$, $G\in \mathbb{R}^{d \times n}$, the target rank $k <n$, and the sample size $l=k+p$.\\
	\hspace*{0.02in} {\bf Output:} 
    $C_A$, $R_A$, $M_A$, 
    $C_B$, $R_B$, $M_B$, 
    $C_G$, $R_G$, and $M_G$.	
	\begin{algorithmic}[1]
		\State Perform Algorithm \ref{alg1} on the matrix pairs $(A,G)$ and $(A^T,B^T)$, respectively, to obtain
  $C_A=A(:,J)$, $C_G=G(:,J)$, $R_A=C_{A^T}^T=A(I,:)$, and $R_B=C_{B^T}^T=B(I,:)$ with $\lvert I \rvert=\lvert J \rvert=l$.
		\State Compute the exact CPQRs of $R_B$ and $C_G^T$, respectively, to obtain 
  $C_B=B(:,J_B)$ and $R_G=G(I_G,:)$ with $\lvert J_B\rvert=\lvert I_G \rvert=l$.			
		\State 	Compute $M_A=C_A^{\dagger}AR_A^{\dagger}$ ,  $M_B=C_B^{\dagger}BR_B^{\dagger}$ and $M_G=C_G^{\dagger}GR_G^{\dagger}$.
	\end{algorithmic}
\end{algorithm}

Noting that $X_1\Pi_{31}=Q_3T_{31}$ and $Y_3^T\Pi_{41}=Q_4T_{41}$ are invertible, we can define two oblique projectors: 
$P_{\Pi_{31},N_7}=\Pi_{31}(X_1\Pi_{31})^{-1}X_1$ and $P_{Y_3,N_8}=Y_3(\Pi_{41}^TY_3)^{-1}\Pi_{41}^T$.
Using them, the following results hold.
\begin{theorem} \label{theorem10}
	Let $C_A\in \mathbb{R}^{m \times l}, C_B\in \mathbb{R}^{m \times l}, C_G\in \mathbb{R}^{d \times l}, R_A\in \mathbb{R}^{l \times n}, R_B\in \mathbb{R}^{l \times t}, R_G\in \mathbb{R}^{l \times n}, M_A\in \mathbb{R}^{l \times l}, M_B\in \mathbb{R}^{l \times l}$ and $M_G\in \mathbb{R}^{l \times l}$ be obtained by Algorithm \ref{alg4}. Then
	\begin{align}
		\Vert A-C_AM_AR_A \Vert &\leq \Vert I_n-P_{\Pi_{31},N_7} \Vert \Vert A(I_n-X_1^{\dagger}X_1) \Vert\nonumber\\
  &+ \Vert I_m-P_{Y_3,N_8} \Vert \Vert (I_m-Y_3Y_3^{\dagger})A \Vert, \label{Aerror} \\ 
		\Vert B-C_BM_BR_B \Vert &\leq \Vert I_m-P_{Y_3,N_8} \Vert \Vert (I_m-Y_3Y_3^{\dagger})B \Vert, \label{Berror}\\
		 \Vert G-C_GM_GR_G \Vert &\leq \Vert I_n-P_{\Pi_{31},N_7} \Vert \Vert G(I_n-X_1^{\dagger}X_1) \Vert. \label{Gerror}
	\end{align}
\end{theorem}
\begin{proof}
	The proof is similar to those of Theorems \ref{theorem1} and  \ref{theorem2}. We first consider the case for the matrix $A$. 
Since $X_1P_{\Pi_{31},N_7}=X_1$ and $P_{Y_3,N_8}Y_3=Y_3$, we have 
\begin{align*}
		\Vert A-C_AC_A^{\dagger}A \Vert &\leq \Vert I_n-P_{\Pi_{31},N_7} \Vert \Vert A(I_n-X_1^{\dagger}X_1) \Vert,\\
  \Vert A-AR_A^{\dagger}R_A \Vert &\leq \Vert I_m-P_{Y_3,N_8} \Vert \Vert (I_m-Y_3Y_3^{\dagger})A \Vert. 	
	\end{align*}
 Considering the fact that $C_AC_A^{\dagger}$ is an orthogonal projector, we get
\begin{align*}
	\Vert A-C_AM_AR_A \Vert &= \Vert A-C_AC_A^{\dagger}AR_A^{\dagger}R_A \Vert \\
	&= \Vert (I_m-C_AC_A^{\dagger})A+C_AC_A^{\dagger}A(I_n-R_A^{\dagger}R_A) \Vert \\
	& \leq \Vert (I_m-C_AC_A^{\dagger})A \Vert+ \Vert C_AC_A^{\dagger}\Vert \Vert A(I_n-R_A^{\dagger}R_A) \Vert \\
	&= \Vert (I_m-C_AC_A^{\dagger})A \Vert+ \Vert A(I_n-R_A^{\dagger}R_A) \Vert.
\end{align*}
Thus, (\ref{Aerror}) can be obtained. Similarly, for the cases of the matrices $B$ and $G$, the following results hold, 
\begin{align*}
		\Vert G-C_GC_G^{\dagger}G \Vert &\leq \Vert I_n-P_{\Pi_{31},N_7} \Vert \Vert G(I_n-X_1^{\dagger}X_1) \Vert,\\  \Vert B-BR_B^{\dagger}R_B \Vert &\leq \Vert I_m-P_{Y_3,N_8} \Vert \Vert (I_m-Y_3Y_3^{\dagger})B \Vert. 
	\end{align*}
Note that the CPQRs of $C_G^T$ and $R_B$ can imply the exact row IDs. Then 
\begin{align*}
	\Vert B-C_BM_BR_B \Vert = \Vert B-BR_B^{\dagger}R_B \Vert, \Vert G-C_GM_GR_G \Vert = \Vert G-C_GC_G^{\dagger}G \Vert.
\end{align*}
As a result, (\ref{Berror}) and (\ref{Gerror}) are derived.
 \end{proof}

Now we present the pass-efficient algorithm, i.e., Algorithm \ref{alg5}, for the CUR decomposition of matrix triplet, whose deduction is  like that of matrix pair. 

\begin{algorithm}[ht] 
	\caption{Pass-efficient algorithm for the CUR decomposition of $(A,B,G)$} 
	\label{alg5}
	\hspace*{0.02in} {\bf Input:} 
	$A\in \mathbb{R}^{m \times n}$, $B\in \mathbb{R}^{m \times t}$, $G\in \mathbb{R}^{d \times n}$, the target rank $k <n$, and the sample size $l=k+p$. \\
	\hspace*{0.02in} {\bf Output:} 
    $C_A$, $R_A$, $M_A$, 
    $C_B$, $R_B$, $M_B$, 
    $C_G$, $R_G$, and $M_G$.
	\begin{algorithmic}[1]
		\State Draw four Gaussian matrices $\Omega_2\in \mathbb{R}^{l \times (m+d)}$, $\Omega_3\in \mathbb{R}^{l \times (n+t)}$, $\Omega_4\in \mathbb{R}^{l \times m}$ and $\Omega_5\in \mathbb{R}^{l \times n}$.
		\State In a single pass, compute
  $X_1=\Omega_2 \begin{bmatrix}
				A \\ G
			\end{bmatrix}$, $Y_3=\begin{bmatrix}
			A, B
		\end{bmatrix}\Omega_3^T$, $X_2=\Omega_4 B$, and $Y_4=G\Omega_5^T$. 		
		\State 	Compute the exact CPQR of $X_1$ to obtain $C_A=A(:,J)$ and $C_G=G(:,J)$ with $\lvert J\rvert =l$.
		\State 	Compute the exact CPQR of $Y_3^T$ to obtain $R_A=A(I,:)$ and $R_B=B(I,:)$ with $\lvert I\rvert =l$.
		\State 	Compute the exact CPQRs of $X_2$ and $Y_4^T$, respectively,  to obtain $C_B=A(:,J_B)$ and $R_G=G(I_G,:)$ with $\lvert J_B\rvert =\lvert I_G\rvert =l$	.	
        \State 	Compute $M_A=C_A^{\dagger}AR_A^{\dagger}$, $M_B=C_B^{\dagger}BR_B^{\dagger}$, and $M_G=C_G^{\dagger}GR_G^{\dagger}$.
	\end{algorithmic}
\end{algorithm}

Next, we present the expectation error bounds  for Algorithm \ref{alg4}.  For brevity, we only present the results of spectral norm.
\begin{theorem}
With the same setting as Theorem \ref{theorem10}, let $k$ be the target rank and $p$ be the oversampling parameter, and let
	\begin{align*}
		\beta &= \left( 1+\sqrt{\frac{k}{p-1}} \right) \sigma_{k+1}\begin{bmatrix}
			A \\ G
		\end{bmatrix} +\frac{{\rm e}\sqrt{k+p}}{p}\left( \sum_{j>k}^{}\sigma_j^2\begin{bmatrix}
			A \\ G
		\end{bmatrix} \right)^{\frac{1}{2}}, \\
	\theta &= \left( 1+\sqrt{\frac{k}{p-1}} \right) \sigma_{k+1}\begin{bmatrix}
		A , B
	\end{bmatrix} +\frac{{\rm e}\sqrt{k+p}}{p}\left( \sum_{j>k}^{}\sigma_j^2\begin{bmatrix}
		A , B
	\end{bmatrix} \right)^{\frac{1}{2}}.
	\end{align*}
	Then 
	\begin{align*}
		\mathbb{E}\Vert A-C_AM_AR_A \Vert &\leq \sqrt{1+(n-k-p)4^{k+p-1}}\times \beta \\ &+ \sqrt{1+(m-k-p)4^{k+p-1}}\times \theta, \\
		\mathbb{E}\Vert B-C_BM_BR_B \Vert &\leq \sqrt{1+(m-k-p)4^{k+p-1}}\times \theta, \\
		\mathbb{E}\Vert G-C_GM_GR_G \Vert &\leq 
		\sqrt{1+(n-k-p)4^{k+p-1}}\times \beta.
	\end{align*}
\end{theorem}
\begin{proof}
	By virtue of  Lemma 3.2 in \cite{dong2021simpler}, we have 
	\begin{align*}
		\Vert I_n-P_{\Pi_{31},N_7} \Vert \leq \sqrt{1+(n-k-p)4^{k+p-1}}, \Vert I_m-P_{Y_3,N_8} \Vert  \leq \sqrt{1+(m-k-p)4^{k+p-1}}.
	\end{align*}
	Thus, combining the fact 
	\begin{align*}
		&{\rm max}\{\Vert A(I_n-X_1^{\dagger}X_1) \Vert, \Vert G(I_n-X_1^{\dagger}X_1) \Vert\} \leq  \left \Vert \begin{bmatrix}
			A \\G
		\end{bmatrix}(I_n-X_1^{\dagger}X_1) \right \Vert, \\
	&{\rm max}\{\Vert (I_m-Y_3Y_3^{\dagger})A \Vert, \Vert (I_m-Y_3Y_3^{\dagger})B\Vert\} \leq \Vert (I_m-Y_3Y_3^{\dagger}) \begin{bmatrix}
		A, B
	\end{bmatrix} \Vert
	\end{align*}
with Theorem \ref{theorem10} and Lemma \ref{lem-expectation} implies the desired results.
\end{proof}

Similar to the error analysis of the CUR decomposition for matrix pair, i.e., Theorem \ref{theorem-err3}, we now consider 
the alternative  expectation error bounds of 
Algorithm \ref{alg4}.
Let the GSVDs of $(A,G)$ and $(A^T,B^T)$ be
\begin{align*}
	\begin{bmatrix}
		A \\ G
	\end{bmatrix}= \begin{bmatrix}
		W_1\Sigma_3H_1^T \\ W_2\Sigma_4H_1^T
	\end{bmatrix},\ \begin{bmatrix}
		A^T \\ B^T
	\end{bmatrix}= \begin{bmatrix}
		W_3\Sigma_5H_2^T \\ W_4\Sigma_6H_2^T
	\end{bmatrix},   
\end{align*}
where $H_1\in \mathbb{R}^{n\times n}$ and $ H_2\in \mathbb{R}^{m\times m}$ are nonsingular, $W_1\in \mathbb{R}^{m\times m}, W_2\in \mathbb{R}^{d\times d}, W_3\in \mathbb{R}^{n\times n},$ and $ W_4\in \mathbb{R}^{t\times t}$ are orthogonal, $\Sigma_3\in \mathbb{R}^{m\times n}, \Sigma_4\in \mathbb{R}^{d\times n}, \Sigma_5\in \mathbb{R}^{n\times m},$ and $ \Sigma_6\in \mathbb{R}^{t\times m}$ are diagonal. 
By using the similar partition method of matrix pair, i.e., 
$W_1=\begin{bmatrix}
    W_{11} , W_{12}
\end{bmatrix}$,  $W_2=\begin{bmatrix}
    W_{21} , W_{22}
\end{bmatrix}$, $W_3=\begin{bmatrix}
    W_{31} , W_{32}
\end{bmatrix}$, $W_4=\begin{bmatrix}
    W_{41} , W_{42}
\end{bmatrix}$,  
\begin{align*}
    \Sigma_3 =\begin{bmatrix}
		\Sigma_{31} & \\ & \Sigma_{32} 
	\end{bmatrix}, \Sigma_4 =\begin{bmatrix}
		\Sigma_{41} & \\ & \Sigma_{42} 
	\end{bmatrix}, \Sigma_5 =\begin{bmatrix}
		\Sigma_{51} & \\ & \Sigma_{52} 
	\end{bmatrix},\textrm{ and } \Sigma_6 =\begin{bmatrix}
		\Sigma_{61} & \\ & \Sigma_{62} 
	\end{bmatrix} 
\end{align*}
with $\Sigma_{31}, \Sigma_{41}, \Sigma_{51}$ and $\Sigma_{61}\in \mathbb{R}^{k\times k}$, 
we have 
\begin{align*}
    \Omega_2 \begin{bmatrix}
		A \\ G
	\end{bmatrix} &= \Omega_2 \begin{bmatrix}
		\hat{W}_{11}\Sigma_{31}+\hat{W}_{21}\Sigma_{41}, \hat{W}_{12}\Sigma_{32}+\hat{W}_{22}\Sigma_{42}
	\end{bmatrix}H_1^T \\ 
 & \approx \begin{bmatrix}
		\Omega_2\hat{W}_{21}\Sigma_{41}, \Omega_2\hat{W}_{12}\Sigma_{32}
	\end{bmatrix}H_1^T, \\
 \Omega_3 \begin{bmatrix}
		A^T \\ B^T
	\end{bmatrix} &= \Omega_3 \begin{bmatrix}
		\hat{W}_{31}\Sigma_{51}+\hat{W}_{41}\Sigma_{61}, \hat{W}_{32}\Sigma_{52}+\hat{W}_{42}\Sigma_{62}
	\end{bmatrix}H_2^T \\
 & \approx\begin{bmatrix}
		\Omega_3\hat{W}_{41}\Sigma_{61}, \Omega_3\hat{W}_{32}\Sigma_{52}
	\end{bmatrix}H_2^T,
\end{align*}
where 
\begin{align*}
    \hat{W}_{11}=\begin{bmatrix}
        W_{11} \\ 0
    \end{bmatrix}, \hat{W}_{12}=\begin{bmatrix}
        W_{12} \\ 0
    \end{bmatrix}, \hat{W}_{21}=\begin{bmatrix}
      0 \\   W_{21} 
    \end{bmatrix}, \hat{W}_{22}=\begin{bmatrix}
      0 \\   W_{22} 
    \end{bmatrix}, \\ 
    \hat{W}_{31}=\begin{bmatrix}
        W_{31} \\ 0
    \end{bmatrix}, \hat{W}_{32}=\begin{bmatrix}
        W_{32} \\ 0
    \end{bmatrix}, \hat{W}_{41}=\begin{bmatrix}
      0 \\   W_{41} 
    \end{bmatrix}, \hat{W}_{42}=\begin{bmatrix}
      0 \\   W_{42} 
    \end{bmatrix}. 
\end{align*}
Let the QR decompositions of $H_1^T$ and $H_2^T$ be $H_1^T=Q_{H_1}R_{H_1}$ and $H_2^T=Q_{H_2}R_{H_2}$ respectively.
Thus, under the following 
assumptions: 
	\begin{itemize}
	    \item [1.] 
     ${\rm range}(R_{H_1}^{-1}(R_{H_1}^{-1})^TX_1^T)\subset {\rm range}(X_1^T)$
     and ${\rm range}(R_{H_2}^{-1}(R_{H_2}^{-1})^TY_3)\subset {\rm range}(Y_3)$;
 \item [2.] $\Omega_2\hat{W}_{21}$ and $\Omega_3\hat{W}_{41}$ are full column rank;
\end{itemize}
combining the GSVD-based error analyses in Section \ref{GSVD based error analyses} with Theorem \ref{theorem10}, we can obtain the following theorem.

\begin{theorem}
	With the same setting as Theorem \ref{theorem10} and the above assumptions, let $k$ be the target rank and $p$ be the oversampling parameter,  
 and let 
 \begin{align*}
     \eta_1=\Vert H_1 \Vert \sqrt{1+(n-k-p)4^{k+p-1}}, \eta_2=\Vert H_2 \Vert \sqrt{1+(m-k-p)4^{k+p-1}}.
 \end{align*}
 Then
 {\small\begin{align*}
		\mathbb{E}\Vert A-C_AM_AR_A \Vert &\leq 
  \eta_1 \left[ \Vert \Sigma_{31}\Sigma_{41}^{-1} \Vert \left( \frac{{\rm e}\sqrt{k+p}}{p}\Vert \Sigma_{32} \Vert_F+\sqrt{\frac{k}{p-1}}\Vert \Sigma_{32} \Vert \right) +\Vert \Sigma_{32} \Vert \right] \\
		&  +\eta_2 \left[ \Vert \Sigma_{51}\Sigma_{61}^{-1} \Vert \left( \frac{{\rm e}\sqrt{k+p}}{p}\Vert \Sigma_{52} \Vert_F+\sqrt{\frac{k}{p-1}}\Vert \Sigma_{52} \Vert \right) +\Vert \Sigma_{52} \Vert \right] , \\
		\mathbb{E}\Vert B-C_BM_BR_B \Vert &\leq \eta_2 \left[ \left( \frac{{\rm e}\sqrt{k+p}}{p}\Vert \Sigma_{52} \Vert_F+\sqrt{\frac{k}{p-1}}\Vert \Sigma_{52} \Vert \right) + \Vert \Sigma_{62} \Vert \right], \\		
		\mathbb{E}\Vert G-C_GM_GR_G \Vert &\leq  \eta_1 \left[ \left( \frac{{\rm e}\sqrt{k+p}}{p}\Vert \Sigma_{32} \Vert_F+\sqrt{\frac{k}{p-1}}\Vert \Sigma_{32} \Vert \right)+ \Vert \Sigma_{42} \Vert \right].
	\end{align*}}
\end{theorem} 
\begin{remark}
    Similar to Theorem \ref{theorem2}, Theorem \ref{theorem-err2} and Corollary \ref{COR1}, we can also obtain two alternative expectation error bounds of Algorithm \ref{alg5}.  For brevity, we omit them here.
\end{remark}

\section{Numerical experiments} \label{sec-num}
In this section, we 
compare numerically the CUR algorithms of matrix pair and matrix triplet whose abbreviations are summarized in Table \ref{table}, among which 
LDEIM-PCUR is obtained by combining the L-DEIM index selection (i.e., \cite[Algorithm 2]{gidisu2022rsvd}) with DEIM-PCUR (i.e., \cite[Algorithm 4.1]{gidisu2022generalized}).
 These algorithms 
 are performed five times, and the numerical results on relative errors and runtime are presented as their average. For a matrix $A$, the relative error of its low rank approximation is defined by $\Vert A-C_AM_AR_A \Vert/ \Vert A \Vert$. 

\begin{table}
\caption{Abbreviations for algorithms}
\label{table}       
\begin{tabular}{lll}
\hline\noalign{\smallskip}
 & matrix pair   & matrix triplet  \\
\noalign{\smallskip}\hline\noalign{\smallskip}
CPQR+Randomized & RPCUR(Algorithm \ref{alg2}) & RTCUR(Algorithm \ref{alg4}) \\
Pass-efficient & PE-PCUR(Algorithm \ref{alg3}) & PE-TCUR(Algorithm \ref{alg5}) \\
DEIM & DEIM-PCUR(\cite[Algorithm 4.1]{gidisu2022generalized}) & DEIM-TCUR(\cite[Algorithm 3]{gidisu2022rsvd}) \\
L-DEIM & LDEIM-PCUR & LDEIM-TCUR(\cite[Algorithm 4]{gidisu2022rsvd}) \\
L-DEIM+Randomized & RLDEIM-PCUR(\cite[Algorithm 5]{cao2023randomized}) & RLDEIM-TCUR(\cite[Algorithm 7]{cao2023randomized})\\
\noalign{\smallskip}\hline
\end{tabular}
\end{table}

\subsection{On the matrix pair}

{\bf Experiment 1} 
As done in \cite{wei2021randomized}, we generate the matrix pair $(A,B)$ as follows, 
\begin{align*}
    A=A_1A_2, \  B=B_1B_2,
\end{align*}
where $A_1\in \mathbb{R}^{10000\times 100}, A_2\in \mathbb{R}^{100\times 5000}, B_1\in \mathbb{R}^{8000\times 100}$ and $ B_2\in \mathbb{R}^{100\times 5000}$ are randomly generated by using the MATLAB build-in function randn. Note that ${\rm rank}(A)={\rm rank}(B)=100$ with probability one. 
We 
implement the algorithms in Table \ref{table} on the above matrix pair $(A,B)$, and present the relative errors and runtime in Fig. \ref{gcur_errors_times2}. 

\begin{figure}[h] 
	\centering
	\includegraphics[scale=0.31]{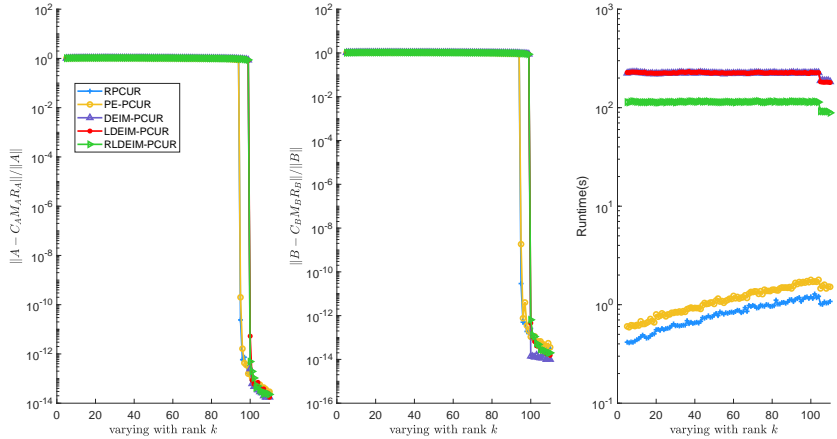} 
	\caption{Errors and runtime for different CUR algorithms of  $(A,B)$}   
	\label{gcur_errors_times2}	    	  	   
\end{figure} 

From the first two figures in Fig. \ref{gcur_errors_times2}, it is seen that if the target rank is less than 100,  all the algorithms have almost the same performance in relative errors. When the target rank approaches 100, the relative errors of RPCUR and PE-PCUR are smaller than those of DEIM-PCUR,  LDEIM-PCUR and RLDEIM-PCUR. This mainly because the oversampling parameter $p=5$ is used in RPCUR and PE-PCUR. Therefore, all the algorithms actually perform almost the same in relative errors.

From the last figure in Fig. \ref{gcur_errors_times2},
we can find that RPCUR needs the least  runtime, followed by PE-PCUR. They are much cheaper than RLDEIM-PCUR, which in turn is cheaper than DEIM-PCUR and LDEIM-PCUR.

\subsection{On the matrix triplet}

{\bf Experiment 2} As done in Experiment 1, we construct the matrix triplet $(A,B,G)$ as follows, 
\begin{align*}
    A=A_1A_2,\  B=B_1B_2,\  G=G_1G_2,
\end{align*}
where $A_1\in \mathbb{R}^{5000\times 100}, A_2\in \mathbb{R}^{100\times 5000}, B_1\in \mathbb{R}^{5000\times 100}, B_2\in \mathbb{R}^{100\times 10000}, G_1\in \mathbb{R}^{10000\times 100}$ and $G_2\in \mathbb{R}^{100\times 5000}$ are generated by the MATLAB function randn. It is clear that ${\rm rank}(A)={\rm rank}(B)={\rm rank}(G)=100$ with probability one. 
Then, the algorithms in Table \ref{table} on the above matrix triplet $(A,B,G)$ are  performed, and the relative errors and  runtime are reported in Fig. \ref{tcur_errors_times2}. 

\begin{figure}[h]
	\centering
	\includegraphics[scale=0.31]{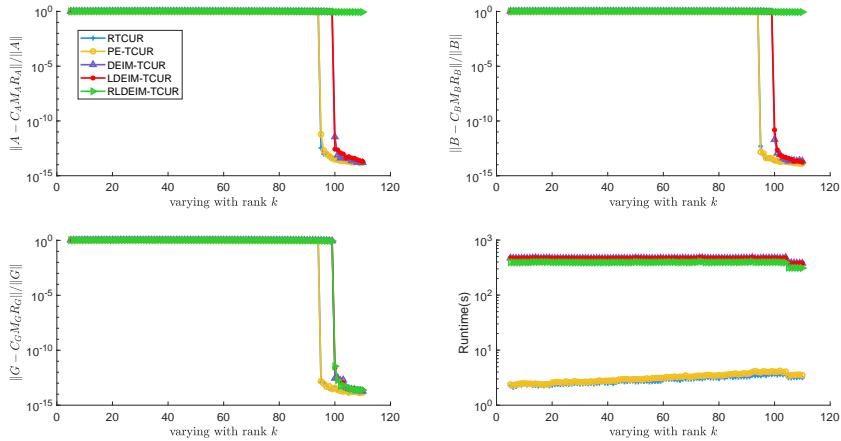} 
	\caption{Errors and runtime for different CUR algorithms of  $(A,B,G)$}    
	\label{tcur_errors_times2}	    	  	   
\end{figure}

From Fig. \ref{tcur_errors_times2}, it follows that the conclusions are similar to those obtained in Experiment 1 except that, in approximating $A$ and $ B$, 
the relative errors of RLDEIM-TCUR are barely falling when $k> 100$. 


\section{Conclusions}
    This paper presents the CPQR-based randomized algorithms for the generalized CUR decompositions of matrix pair and matrix triplet. Their pass-efficient versions are also obtained. 
    For these algorithms, we present two alternative expectation error analyses, among which 
    the second one is mainly inspired by 
    the method in \cite{halko2011finding}. 
    Numerical experiments demonstrate that our algorithms can calculate the generalized CUR decompositions as accurate as the existing methods, but need much less  runtime. 

%
%




\end{document}